\documentclass[oneside,11pt]{amsart}
\usepackage[utf8]{inputenc}
\usepackage[english]{babel}
\usepackage{libertine}
\usepackage{floatflt}
\usepackage{comment}
\usepackage{blindtext}
\usepackage{enumitem}
\usepackage{amsthm}
\usepackage{subcaption}
\usepackage{listings}
\usepackage{amscd}
\usepackage{centernot}
\usepackage{mathtools}
\usepackage{stmaryrd}
\usepackage{listingsutf8}
\usepackage{setspace}
\usepackage{amsmath}
\usepackage{framed}
\usepackage{minibox}
\usepackage{float}
\usepackage{wrapfig}
\usepackage{longtable}
\usepackage[strict]{changepage}
\usepackage{nicefrac}
\usepackage{bbm}
\usepackage{amssymb}
\usepackage{amsthm}
\usepackage{mathtools}
\usepackage{amsfonts}
\usepackage{mathrsfs}
\usepackage{todonotes}
\usepackage{esvect}
\usepackage{scalerel}
\usepackage{stackengine,wasysym}

\usepackage[toc,page]{appendix}
\usepackage{ragged2e}

\makeatletter
\renewenvironment{thebibliography}[1]
     {\section*{\refname}%
      \normalfont\normalsize % <-- mantiene la dimensione normale
      \list{\@biblabel{\@arabic\c@enumiv}}%
           {\settowidth\labelwidth{\@biblabel{#1}}%
            \leftmargin\labelwidth
            \advance\leftmargin\labelsep
            \usecounter{enumiv}}%
      \sloppy\clubpenalty4000\widowpenalty4000%
      \sfcode`\.=\@m}
     {\def\@noitemerr{\@latex@warning{Empty `thebibliography' environment}}%
      \endlist}
\makeatother
\usepackage{tikz-cd}
\usetikzlibrary{positioning, angles}
\usetikzlibrary{arrows}
\usetikzlibrary{decorations.pathmorphing, decorations.pathreplacing}
\usetikzlibrary{decorations.markings}
\usepackage{mathtools}
\usepackage{lipsum}
\usepackage{xfp}

\usepackage{tikz}
\usepackage{tikz-3dplot}
\usetikzlibrary{3d, angles, arrows.meta, calc, decorations.markings,
                decorations.pathmorphing, decorations.pathreplacing,
                positioning, shapes,shadings}
\usepackage{pgfplots}
\usepackage{pgfmath}
\usepackage{xfp}

\usepackage{hyperref}
\tikzset{
  on each segment/.style={
    decorate,
    decoration={
      show path construction,
      moveto code={},
      lineto code={
        \path [#1]
        (\tikzinputsegmentfirst) -- (\tikzinputsegmentlast);
      },
      curveto code={
        \path [#1] (\tikzinputsegmentfirst)
        .. controls
        (\tikzinputsegmentsupporta) and (\tikzinputsegmentsupportb)
        ..
        (\tikzinputsegmentlast);
      },
      closepath code={
        \path [#1]
        (\tikzinputsegmentfirst) -- (\tikzinputsegmentlast);
      },
    },
  },
  mid arrow/.style={postaction={decorate,decoration={
        markings,
        mark=at position .5 with {\arrow[#1]{stealth}}
      }}},
}
\usepackage[toc,page]{appendix}
\usepackage{ragged2e}
\usepackage{fancyhdr}

\usepackage{fancyhdr}\oddsidemargin = 2pt       \topmargin = 0pt\headheight = 12pt          \headsep = 25pt\textheight = 609pt         \textwidth = 424pt\marginparsep = 11pt        \marginparwidth = 54pt\footskip = 30pt            \marginparpush = 5pt \hoffset = 0pt              \voffset = 0pt\paperwidth = 597pt         \paperheight = 800pt
\def\Xint#1{\mathchoice
  {\XXint\displaystyle\textstyle{#1}}% 
  {\XXint\textstyle\scriptstyle{#1}}% 
  {\XXint\scriptstyle\scriptscriptstyle{#1}}% 
  {\XXint\scriptscriptstyle\scriptscriptstyle{#1}}% 
  \!\int}
\def\XXint#1#2#3{{\setbox0=\hbox{$#1{#2#3}{\int}$}
  \vcenter{\hbox{$#2#3$}}\kern-.5\wd0}}
\def\dashint{\Xint-}
\usepackage[normalem]{ulem}

\newcommand{\R}{\ensuremath{\mathbb{R}}}

\newcommand{\norm}[1]{{\ensuremath{||{#1}||}}}

\let\theta\vartheta
\let\phi\varphi
\let\epsilon\varepsilon
\let\bar\overline

\theoremstyle{plain}
\newtheorem{thm}{Theorem}[section]
\newtheorem*{thm*}{Theorem}
\newtheorem{prop}[thm]{Proposition}

\newtheorem{lem}[thm]{Lemma}

\theoremstyle{definition}
\newtheorem*{defn*}{Definition}

\newtheorem{oss}[thm]{Remark}

\newread\tmp

\theoremstyle{definition} % testo in romano

\usepackage{etoolbox}
\makeatletter
\providecommand{\subtitle}[1]{% add subtitle to \maketitle
  \apptocmd{\@title}{\vspace*{0.2cm}\par {\small\normalfont #1 \par}}{}{}
}
\makeatother
\usepackage{tcolorbox}
\counterwithout*{equation}{section}
\def\pri{\hbox to 10pt{\hfil\hbox to 0.4pt{\vrule height5pt width0.4pt
                 depth0pt}\vrule width5pt height0.4pt depth0pt\hfil}}
\usepackage{units}

\addto\captionsenglish{}
\makeatletter
\newcommand{\tpitchfork}{%
  \vbox{
    \baselineskip\z@skip
    \lineskip-.52ex
    \lineskiplimit\maxdimen
    \m@th
    \ialign{##\crcr\hidewidth\smash{$-$}\hidewidth\crcr$\pitchfork$\crcr}
  }%
}
\makeatother

\newcommand{\gu}{{\bf u}}
\newcommand{\gc}{{\bf c}}
\usepackage{xcolor}
\usepackage{graphicx}
\usepackage{listings}
\usepackage{inconsolata}  % Font monospazio (opzionale)
\numberwithin{equation}{section}

\date{}
\usepackage{tikz}
\usetikzlibrary{3d, calc, shadings}

\usepackage{marginnote}
\usepackage{tikz}
\usetikzlibrary{angles,quotes,arrows.meta,calc,intersections}

\title[Elastic energy of curves in Riemannian surfaces]{Weak elastic energy of rectifiable curves in Riemannian surfaces}

\author[D. Mucci]{Domenico Mucci}
\address{Domenico Mucci\\ Dipartimento  di Scienze Fisiche, Informatiche e Matematiche, Università degli studi di Modena e Reggio Emilia\\
Via G. Campi 213/A, 41125 Modena, Italy
}
\email{\url{domenico.mucci@unimore.it}}

\author[A. Saracco]{Alberto Saracco}
\address{Alberto Saracco\\ Dipartimento di Scienze Matematiche, Fisiche e Informatiche, Universit\`a di Parma\\ Campus - Parco Area delle Scienze 53/A, 43124 Parma, Italy}
\email{\url{alberto.saracco@unipr.it}}

\author[C. Sopio]{Cristian Sopio}
\address{Cristian Sopio\\Dipartimento di Matematica e Informatica, Università degli Studi di Ferrara \\ \newline
Via Machiavelli 35, 44121 Ferrara, Italy}  \email{\url{cristian.sopio@unife.it}}
\makeatother
\begin{document}
\nocite{MS2}
\onehalfspacing
\maketitle
\begin{abstract}
We introduce a weak elastic energy for rectifiable curves on compact orientable smooth
  Riemannian surfaces without boundary.
The energy is defined by relaxation starting from a notion of $p$-rotation of inscribed geodesic polygonals, that is obtained by a local construction in normalized isothermal coordinates.
For every exponent $p>1$, the resulting relaxed functional detects precisely the intrinsic second-order Sobolev regularity of the arc-length parameterization of the curve.
Furthermore, when the relaxed energy is finite, it agrees with the integral of the $p$-power of the geodesic curvature.
\end{abstract}
\section{Introduction}
Let $(S,g)$ be a Riemannian surface of class $C^2$, that is assumed to be compact, orientable, and without boundary.
In this paper we introduce a relaxed notion of elastic energy for rectifiable curves $c$ supported in $S$. Extending results obtained in \cite{MSS25} in the case where $S$ is the unit 2-sphere $\mathbb S^2$ in $\mathbb R^3$, we show that for every real exponent $p>1$, our $p$-energy is finite if and only if the arc-length parameterization of $c$ is in the second order Sobolev class $W^{2,p}$, and in that case it agrees with the integral of the $p$-power of the geodesic curvature.

In order to state our main result, we sketch here the main tools of our construction, and we first recall that for any real exponent $p\geq 1$, the total $p$-curvature of a smooth rectifiable curve $c$ in $S$ is given by the integral
\begin{equation}
\label{eq:p-integral}
\int_c |k_S(c)|_g^p\,ds \,,
\end{equation}
where $k_S(c)$ is the curvature vector of $c$ and $|k_s(c)|_g$ is its norm with respect to the metric $g$.

For non-smooth rectifiable curves $c$, our definition of $p$-energy is based on a relaxation process.
The main task, in the elastic case $p>1$ considered here, is to give a robust notion of {\em $p$-energy} for polygonal curves $P$ in $S$, with the aim of spreading the curvature at the edges along a piecewise curve $\gamma(P)$ inscribed in $P$.

For equilateral polygonal curves $P$ into the 2-sphere, this was done in \cite{MSS25} by connecting the middle points of consecutive segments of $P$ with arcs of constant geodesic curvature satisfying a first order condition at the end points.
However, for a Riemannian surface as above, an explicit formula of $\gamma(P)$ fails to hold.

To overcome this difficulty we assume that the mesh of $P$ is small, and we make use of local isothermal coordinates $\varphi_p$ centered at every internal vertex $p$ of $P$.
We then analyze the arcs $\gamma_1$, $\gamma_2$ in $\mathbb R^2$ obtained by projecting half of the two geodesic segments of $P$ insisting at the vertex $p$.

The main issue is that we are able to control their length and the angle between the directions at the end points in terms of the side length of $P$ and of the angle of $P$ at the edge point $p$.

This feature allows us to construct a suitable  arc $\Gamma$ connecting to the end points of $\gamma_i$ at first order, in such a way that its total curvature is comparable to the turning angle, i.e., the angle in $(0,\pi)$ between the inward and outward tangent direction of $P$ at $p$. Then, the pull-back of $\Gamma$ by $\varphi_p$ defines a portion of the piecewise smooth curve $\gamma(P)$ which lies in $S$ and is inscribed in  $P$.

We thus call {\em $p$-rotation} $\mathbf{k}_p(P)$ of $P$ the total $p$-curvature of its inscribed curve $\gamma(P)$, that is
\begin{equation}
\label{p-rotation}
\mathbf{k}_p(P) :=\int_{\gamma(P)} |k_S(\gamma(P))|_g^p\,ds \,.
\end{equation}
Proposition \ref{PgammaP} contains an explicit formula of the $p$-rotation \eqref{p-rotation}. Notice that for $p=1$ we recover the standard geodesic rotation of $P$, that is the sum of the turning angles of $P$.

For a given rectifiable curve $c$ in $S$, we denote by $\mu_c(P)$ the {\em modulus} of a polygonal curve $P$ in $S$ inscribed in $c$, say $P\ll c$.
It is given by the maximum of the geodesic diameter of the pieces of arcs of $c$ determined by couples of consecutive vertices of $P$,  see \cite{alexandrov2012general}. 

In the same spirit as Lebesgue--Serrin's relaxed functional, we define the {\em $p$-curvature} functional $\mathcal{F}_p(c)$ of $c$ as
$$
\mathcal{F}_p(c):=\inf\Bigl\{\liminf_{h\to\infty}\mathbf{k}_p(P_h)\mid \{P_h\}\ll c\,,\,\,\mu_c(P_h)\to 0\Bigr\}\quad\mbox{ for }p\geq 1\,.
$$
Notice that in case $p=1$, the functional $\mathcal{F}_p(c)$ agrees with the total curvature $TC_S(c)$ of $c$ as defined in \cite{MS2, MS2b}.

We also recall that a rectifiable curve $c$ parameterized in arc-length belongs to the second order (intrinsic) Sobolev class $W^{2,p}([0,L],S)$, where $L=\mathcal L(c)$ is the length of $c$, provided that the unit tangent vector $\dot c:[0,L]\to TS$ is differentiable almost everywhere on $I$ and the weak covariant derivative
$$k_S(c):=\nabla_{\dot u} \dot u $$
belongs to $L^p([0,L],TS)$. In that case, it defines the curvature vector at $c(s)$ for almost
every $s\in[0,L]$, and the integral at the right-hand side of equation \eqref{eq:p-integral} is well-defined and finite.

We are now in position to present the main result of this paper.
\begin{thm}\label{MainThm}
     Let $(S,g)$ be a Riemannian surface as above, and let $c:[0,L]\to S$ be a rectifiable and open curve parameterized by arc length. Then $$\mathcal{F}_p(c)<\infty \mbox{ for some }p>1\Longleftrightarrow c\in W^{2,p}([0,L],S)$$ and in this case, there holds
$$
\mathcal{F}_p(c)=\int_c|k_S(c)|_{g}^p\,ds.
$$
\end{thm}

The proof of Theorem \ref{MainThm} follows a similar strategy to the one we used in \cite{MSS25}. The main new contribution is given by
Theorem \ref{Teo1}, where we obtain suitable estimates concerning the
$p$-rotation $\mathbf{k}_p(P)$ of polygonals inscribed in the curve $c$ and with small mesh.\vspace{0.3cm}

{\bf Acknoweledgment.}
The authors were partially supported by the GNAMPA of INDAM.
%%%%%%%%
\section{Notation and preliminary results}
Throughout this paper, we let $(S,g)$ be a Riemannian surface of class $C^2$, that is assumed to be compact, orientable, and without boundary.

We consider curves $c:I\to S$ defined in a nondegenerate interval $I\subset \mathbb{R}$.
A curve $c$ is called rectifiable if its length ${\mathcal L}(c)$ is finite. In that case, its arc-length parameterization is Lipschitz continuous.

The {\em Fr\'echet distance} $d_{F}(c_1,c_2)$ between two rectifiable curves $c_1,c_2:I\to S$ is the infimum, over all strictly monotonic
reparametrizations, of the maximum pointwise geodesic distance $d_S$ between them.
More precisely, if $\beta_1,\beta_2$ are reparametrizations of the interval $I$, we have $$d_F(c_1,c_2)\coloneqq\inf_{\beta_1,\beta_2}\max_{t\in I} d_{S}\Big(c_1(\beta_1(t)),c_2(\beta_2(t))\Big)\,. $$
Moreover, if $\{c_h\} $ is a sequence of rectifiable curves in $S$ such that $d_{F}(c_h,c)\to 0$ as $h\to \infty$ for some rectifiable curve $c$, then by lower semicontinuity
\begin{equation}
    \mathcal{L}(c)\leq\liminf_{h\to\infty}\mathcal{L}(c_h)\,.
\end{equation}
\subsection{Sobolev functional spaces}
We fix an isometric embedding ${\bf j}:S\hookrightarrow \R^m .$
For a given nondegenerate interval $I\subset \mathbb{R}$, and a continuous function $u:I\to S$, we let $\gu={\bf j}\circ u$, and we consider for $k=1,2$ and $p\geq 1$ the Sobolev class
$$
W^{k,p}(I,{\bf j}(S))\coloneqq\left\{\gu\in W^{k,p}(I,\mathbb{R}^m): \gu(t)\in {\bf j}(S) \text{ for a.e. } t\in I\right\}.
$$

Since we deal with intrinsic results, we notice that if $\gu\in W^{2,p}(I,{\bf j}(S))$, the weak covariant derivative $\nabla_{\dot u} \dot u$ is well defined by the unique vector field
$$
\nabla_{\dot u} \dot u\in L^p(I,TS), \qquad \nabla_{\dot u} \dot u(t)\in T_{u(t)}S \ \text{for a.e. } t\in I,
$$
where $TS$ is the tangent bundle, such that
$$
\int_I g(\dot u,\nabla_{\dot u}\Phi)\,dt
=
-\int_I g(\nabla_{\dot u} \dot u,\Phi)\,dt
$$
for every smooth vector field $\Phi$ along $u$ with compact support in $I$.
We then may define
$$
W^{2,p}(I,S)\coloneqq\left\{u:I\to S \mid \gu\in W^{1,p}(I,S), \nabla_{\dot u} \dot u\in L^p(I,TS)\right\}.
$$

If, in addition, $|\dot u|_g=1$ a.e. on $I$, the weak geodesic curvature vector is defined by
$$
k_S(u):=\nabla_{\dot u} \dot u \,,
$$
and in that case one has $|k_S(u)|_g\in L^p(I).$
We correspondingly have $\Vert \dot \gu \Vert=1$ and
$$ |k_S(u)(t)|_g=\Vert \ddot\gu(t)^\top \Vert $$
for a.e. $t\in I$, where $\ddot\gu(t)^\top$ is the orthogonal projection of $\ddot\gu(t)\in\R^m$ onto the tangent plane to ${\bf j}(S)$ at $\gu(t)$. Notice moreover that the normal component $\ddot\gu(t)^\perp:= \ddot\gu(t)-\ddot\gu(t)^\top$ is uniformly bounded by a constant only depending on  the second fundamental form of the embedding $\mathbf{j}$.

In conclusion, if a rectifiable curve $c:I\to S$ is parameterized in arc-length, we infer that
$c\in W^{2,p}(I,S)$ if and only if $|k_S(c)|_g\in L^p(I).$

\subsection{Local isothermal coordinates}
The following result specializes well-known properties of Riemannian surfaces as above to our purposes.
\begin{thm}\label{isothermal _coordinates}
For every $\epsilon>0$, there exist two positive real numbers $\delta_1=\delta_1(\epsilon,S)>0$ and $C_0>0$ such that for every $p\in S$ we can find a neighborhood $U_p$ of $p$ in $S$ of positive diameter $\delta$ lower than $\delta_1$, and a conformal map $\phi_p: U_p\to \R^2  $, that is
$$
g_q=e^{2\lambda_p(\phi_p(q))}\mathrm{Id}\quad\forall\,q\in U_p\,, $$
such that $\phi_p(p)=(0,0)$ and the function $\lambda_p:\phi_p(U_p)\to\R$ satisfies:
    \begin{equation}
\label{normalizzazione}
        \lambda_p(0,0)=0\,,\quad \nabla\lambda_p(0,0)=(0,0)
    \end{equation}
\begin{equation}
\label{stimeIsoterme}
|e^{\lambda_p(\phi_p(q))}-1|\leq \epsilon \,,
         \quad \|\nabla\lambda_p ( \phi_p(q))\|\leq C_0\delta \quad\forall\, q\in U_p\,,
    \end{equation}
and finally the second order mixed derivatives satisfy
    \begin{equation}\label{missta}
\partial^2_{xy}\lambda_p(0,0)=\partial^2_{yx}\lambda_p(0,0)=0.
\end{equation} 

\end{thm}

\begin{proof}
The existence of isothermal coordinates for a Riemannian manifold is a well known fact, see e.g. \cite{Ch, DK}.

The normalization conditions \eqref{normalizzazione} are obtained by using that the system yielding to the existence of local isothermal coordinates is overdetermined. More specifically, it is sufficient to post-compose with a conformal coordinate change. 
First scale the coordinates 
so that
$\lambda_p(0,0)=0$ and then use a local holomorphic change of the form $w=f(z)=z+az^2$. Choosing $a$ appropriately, one gets $\nabla \lambda_p(0,0)=(0,0)$. 

Furthermore, property \eqref{missta} follows by applying the spectral theorem for linear maps to the Hessian of $\lambda_p$ at $(0,0)$ and observing that conjugating by a rotation preserves the first order conditions of $\lambda_{p}$ at $(0,0)$.

The first inequality in \eqref{stimeIsoterme} follows since the function $\lambda_p$ is equi-uniformly continuous with respect to the base point $p\in S$, as a consequence of the  smoothness and compactness of $S$.

Furthermore, we can bound the second order expansion of the conformal factor $\lambda_p$. Indeed, we have
$$
\lambda_p(x,y)=\underbrace{\lambda_p(0,0)}_{=0}+\underbrace{\nabla\lambda_p(0,0)}_{=(0,0)}(x,y)+\frac 12\,(x,y) D^2\lambda_p(0,0)(x,y)^\top+o(|(x,y)|^2),
$$
which implies
$$
\nabla\lambda_p(x,y)= (x,y)\,D^2\lambda_p(0,0)+o(|(x,y)|).
$$
We observe that
\begin{equation}\label{Heis}
        |D^2\lambda_p(\phi_p(q))| \leq C_0,\quad \forall q\in U_p.
    \end{equation}
Indeed, the estimate follows from a standard compactness argument, observing that the derivatives of $\lambda_p$ are functions of Christoffel symbols and their derivatives, see \eqref{simbolielambda}.
We thus have the second estimate in \eqref{stimeIsoterme} and the proof is complete.
\end{proof}
\begin{oss} In the case when $S$ is the unit sphere $\mathbb S^2$ in $\mathbb R^3$, in \cite{MSS25} we chose $\phi_p$ equal to a rescaled stereographic projection from the antipodal point to $p$, so that
$$
\lambda_p(x,y)=\log\frac 4{4+x^2+y^2}\quad\forall\,p\in\mathbb S^2 .
$$
With this choice, the properties \eqref{normalizzazione}, \eqref{stimeIsoterme} and \eqref{missta}
are readily checked.
\end{oss}
\subsection{Geodesic curvature via conformal maps}
\par We shall also make use of the following estimate, the proof of which is omitted.
 \begin{prop}[\cite{MSS25}]\label{stima_p}
    For every $p\geq 1$ there exists a real constant $C=C(p)>0$ such that for every $a,b\geq 0$
 \begin{equation}\label{dis-alto}
   a^p-Ca^{p-1}b\leq |a+b|^p \leq a^p+Ca^{p-1}b+Cb^p
\end{equation}
\end{prop}
The following lemma is proved in \cite{MSS25}.

\begin{lem}\label{LemmaComparison}
    Let $f:(M,g)\rightarrow (N,g')$ be a conformal map between $2$--dimensional oriented Riemannian manifolds such that $g=e^{2\lambda }g'$.
Let $c:[0,L]\rightarrow M$ be a curve of class $C^2$, and suppose that the image curve $\gamma\coloneqq f\circ c:[0,L]\rightarrow N$ is parameterized by arc length and let $\hat\gu\in T_{\gamma}N$ be the conormal to $\gamma$.
Then, if $k_M$ and $k_{N}$ are the geodesic curvature of $c$ and $\gamma$, respectively, for every $t\in (0,L)$ there holds \begin{equation}\label{GCconformal}
        k_{M}(t)=e^{-\lambda(\gamma(t))}\Big(k_{N}(t)+\partial_{\hat\gu}\lambda(\gamma(t))\Big).
    \end{equation}
\end{lem}
In Section \ref{Sec-gamma(P)}, we will use the Lemma above using 
{$M=S$, $N=\R^2$} and replacing $f$ with $\phi_p$ for some $p\in S$. 
Then we have 
%curve $\gamma $ will be the preimage of a curve $c$ in $S$, namely 
$\gamma=\phi_p(c)$.

In particular, what we will need in Section \ref{Sec-gamma(P)} is an integral consequence of {formula} \eqref{GCconformal}. Indeed, if $c\in W^{2,p}$, passing through the arc-length parametrization of $\gamma$, we have 
\begin{equation}
\label{intcurv}
     \int_0^L\vert k_{S}(c(s)) \vert_{g}^p\, ds=\int_0^{\mathcal{L}(\gamma)}e^{(1-p)\lambda_p(\gamma(s))}\left\vert k_{\R^2}(\gamma(s)) +\partial_{\hat\gu}\lambda_p(\gamma(s)) \right\vert^p
\, ds \,.
    \end{equation}

\subsection{Geodesic expansion} Using the geodesic equation, we compute the Taylor expansion of the image in $\R^2$ by the map $\phi_p$ from Theorem \ref{isothermal _coordinates} of a small geodesic arc in $S$ starting at $p$.

We consider a geodesic curve $\gamma:[0,\delta]\to S$ parameterized by arclength with
$$
\gamma(0)=p\in S,\qquad  v=\dot\gamma(0),\qquad%\norm{v}_{T_pS}\coloneqq 
|v|_g=1.
$$
We call $\gamma_p\coloneqq \phi_p\circ \gamma$, where $\phi_p$ is the map of Theorem \ref{isothermal _coordinates}. Therefore, we assume $\delta\leq \delta_1$.
Since $\gamma_p(t)=(x(t),y(t))$ is an ordinary curve in $\mathbb{R}^2$, it admits the
Taylor expansion:
\begin{equation}
\label{Taylor}
\begin{cases}
    x(t) = x(0) + t\,\dot{x}(0) + \frac{t^2}{2}\ddot{x}(0)+o(t^2) \\
y(t) = y(0) + t\,\dot{y}(0) + \frac{t^2}{2}\ddot{y}(0)
      + o(t^2).
\end{cases}
\end{equation}

This expansion is purely Euclidean; the geometry of $S$ enters through
the relations satisfied by the derivatives, since $\gamma$ is a geodesic. Indeed, in local charts, the geodesic equation
$$%\begin{equation}
\ddot{x}^k + \Gamma^k_{ij}\dot{x}^i\dot{x}^j = 0
$$%\end{equation}
gives a geometric meaning to the second derivatives of the curve, i.e., $$\ddot{x}^1 =-\Gamma^1_{ij}\dot{x}^i\dot{x}^j,\qquad\ddot{x}^2 =-\Gamma^2_{ij}\dot{x}^i\dot{x}^j.$$
For the conformal metric $$g_{ij} = e^{2\lambda} \delta_{ij},
\qquad
g^{ij} = e^{-2\lambda} \delta^{ij},$$ the Christoffel
symbols are
$$
\Gamma^k_{ij}
=
\frac{1}{2} g^{k\ell}
\left(
\partial_i g_{j\ell}
+
\partial_j g_{i\ell}
-
\partial_\ell g_{ij}
\right),
$$
$$\partial_i g_{j\ell}
=
2 e^{2\lambda} (\partial_i \lambda_p)\, \delta_{j\ell},$$
$$
\Gamma^k_{ij}
=
\delta^k_j \partial_i\lambda_p
+
\delta^k_i \partial_j\lambda_p
-
\delta_{ij} \partial^k\lambda_p.
$$
Writing $\lambda_x = \partial_x\lambda_p$ and $\lambda_y = \partial_y\lambda_p$, we have
%\begin{equation}\label{simboli_conformi}
\begin{equation}\label{simbolielambda}
\begin{aligned}
\Gamma^1_{11} &= \lambda_x,
&
\Gamma^1_{12} = \Gamma^1_{21} &= \lambda_y,
&
\Gamma^1_{22} &= -\lambda_x,
\\
\Gamma^2_{22} &= \lambda_y,
&
\Gamma^2_{12} = \Gamma^2_{21} &= \lambda_x,
&
\Gamma^2_{11} &= -\lambda_y,
\end{aligned}   
\end{equation}
%\end{equation}
so that the geodesic equation becomes
$$\begin{cases}
\ddot{x}(t)
=
-\big(
\lambda_x(\gamma(t))\,(\dot{x}(t)^2-\dot{y}(t)^2)
+ 2\lambda_y(\gamma(t)) \,\dot{x}(t)\dot{y}(t)
\big) \\
\ddot{y}(t)
=
-\big(
\lambda_y(\gamma(t))\,(\dot{y}(t)^2-\dot{x}(t)^2)
+ 2\lambda_x(\gamma(t))\, \dot{x}(t)\dot{y}(t)
\big).
\end{cases}
$$
%Since $(d\phi_p)_p\dot\gamma(0)=(d\phi_p)_p v=(\dot{x}(0),\dot{y}(0))$, 
Evaluating the geodesic equation at $t=0$, so that $\gamma(0)=(x(0),y(0))=(0,0)$,
and substituting into the Euclidean Taylor expansion yields
$$\begin{cases}
    x(t)
=t \dot x(0)
-\frac{t^2}{2}
\big(
\lambda_x(0,0)\,(\dot x(0)^2-\dot y(0)^2)
+2\lambda_y(0,0)\,\dot x(0) \dot y(0)
\big)
+ o(t^2) \\
y(t)
=t \dot y(0)
-\frac{t^2}{2}
\big(
\lambda_y(0,0)\,(\dot y(0)^2-\dot x(0)^2)
+2\lambda_x(0,0)\,\dot x(0)\dot y(0)
\big)
+ o(t^2).
\end{cases}$$ Then we have
$$
    \gamma_p(t)= t\,(d\phi_p)_p v-\dfrac{t^2}{2}A_p+o(t^2),
$$
where $$A_p=\left(\begin{array}{c}
\lambda_x(0,0)(\dot x(0)^2-\dot y(0)^2)
+2\lambda_y(0,0)\dot x(0) \dot y(0) \\
%\,,\,
\lambda_y(0,0)(\dot y(0)^2-\dot x(0)^2)
+2\lambda_x(0,0)\dot x(0)\dot y(0)
\end{array}\right).
$$

In conclusion, since by the second equation in \eqref{normalizzazione} the vector $A_p$ is zero, we obtain:
% since the gradient of $\lambda$ is zero at $(0,0)$ and 
\begin{equation}\label{expansion}
    \gamma_p(t)= t\,(d\phi_p)_p v+o(t^2).
\end{equation}
%
%%%%%
%
\section{Piecewise smooth curves inscribed in polygonals}\label{Sec-gamma(P)}
We wish to construct a suitable piecewise smooth curve $\gamma(P)$ linked to a polygonal $P$ in such a way that the definition \eqref{p-rotation} of $p$-rotation makes sense.

To this purpose, in this section we consider two consecutive geodesic arcs of the polygonal with junction point in $p\in S$ and turning angle $\theta$, and
%. To obtain a suitable estimate of the $p$--rotation in \eqref{p-rotation},
% Theorem \ref{Teo1}, 
we look the edges and the angles of the triangle in charts where we want to inscribe a piece of the curve $\gamma(P)$.
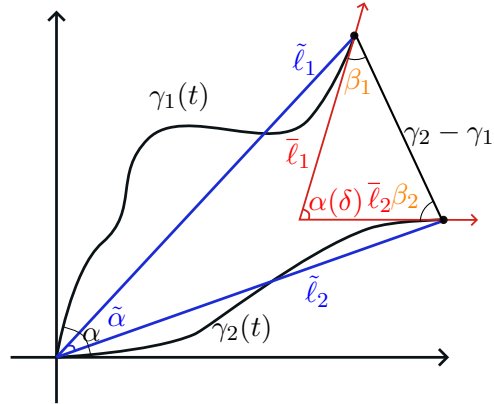
\begin{figure}[H]
    \centering
    \begin{tikzpicture}[scale=0.83]
    %\draw[step=1cm,color=gray] (0,0) grid (15,6);%Uncomment this to get some helpful grid lines
    \node[anchor=south west,inner sep=0] at (0,0){\includegraphics[width=0.415\textwidth]{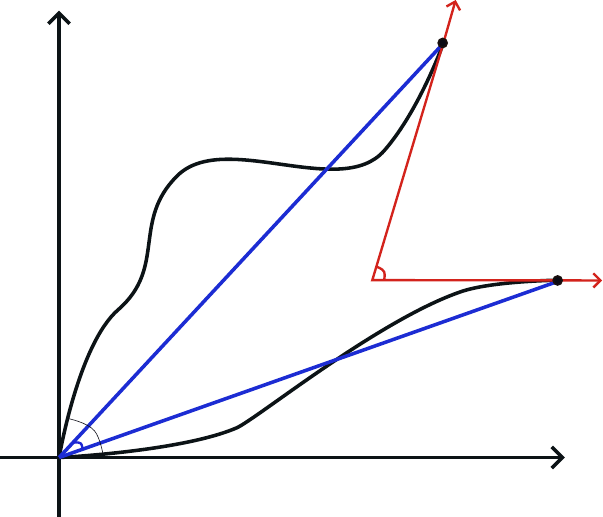}};
       \node at (2.7,4.9) {$\gamma_1(t)$};
       \node at (3.7,1.2) {$\gamma_2(t)$};
    \node at (1.3,1.1) {$\alpha$};
    \node at (5.2,3.2) {$\textcolor{red}{\alpha(\delta)}$};
     \node at (4.6,4) {$\textcolor{red}{\bar\ell_1}$};
     \node at (4.7,5.5) {$\textcolor{blue}{\tilde\ell_1}$};
     \node at (4.9,1.8) {$\textcolor{blue}{\tilde\ell_2}$};
     \node at (5.9,3.3) {$\textcolor{red}{\bar\ell_2}$};
     \node at (1.7,1.4) {$\textcolor{blue}{\tilde\alpha}$};
 \draw[black, thick] (5.5,5.9) to (6.9,2.9);
 \node at (7,4.3) { $\gamma_2-\gamma_1$ };

  % \node at (5.45,5.85) {$A$ };
   %\node at (4.6,3) {$B$ };
    %  \node at (6.9,2.9) {$C$ };
\coordinate (A) at (5.5,5.85);
\coordinate (B) at (4.6,3);
\coordinate (C) at (6.9,2.9);
\draw pic[draw=black,-,"\textcolor{orange}{$\beta_1$}",angle eccentricity=2,angle radius=0.3cm]
{angle = B--A--C};
\draw pic[draw=black,-,"\textcolor{orange}{$\beta_2$}",angle eccentricity=2,angle radius=0.3cm]
{angle = A--C--B};
    \end{tikzpicture}
    \caption{Picture in $\R^2$ after the conformal map $\phi_p$.}
    
    \label{figura}
    \end{figure}
    \subsection{Turning angles and edges}
We consider the piece of the polygonal $P$ given by two consecutive geodesic arcs $\gamma_1,\gamma_2$  of small length $\delta$, parameterized by arc length, with initial velocity $v_1,v_2\in T_pS$ and turning angle $\theta $ at $p$, and we call $\alpha=\pi-\theta$ the angle between $v_1,v_2\in T_pS$.

We consider the image curves through the map $\phi_p$,
$$ \gamma_{i,p}=\phi_p\circ\gamma_i\,,\quad i=1,2\,,
$$
and we analyze in three steps its main properties.
    \begin{oss}
       Figure \ref{figura} shows the blue and red triangles of interest obtained after projecting two consecutive geodesic arcs onto $\R^2$.

In particular, when the arcs belong to a polygonal curve $P$ inscribed in a 
sufficiently smooth curve $c$, for $\mu_c(P)$ small we have that the angle $\alpha=\pi-\theta$ is almost equal to $\pi$.
    \end{oss}
%
%\smallskip
\par\noindent
{\em Step 1.}
We first consider the function $t\mapsto \alpha(t)$ that gives the angle between the tangent vectors $\dot\gamma_{i,p}$ at time $t$, i.e.
$$
\cos\alpha(t)=\dfrac{\langle\dot \gamma_{1,p}(t),\dot\gamma_{2,p}(t)\rangle_{\R^2}}{\|\dot \gamma_{1,p}(t)\|\|\dot \gamma_{2,p}(t)\|}.
$$
By conformality, using equation \eqref{expansion} we have
\begin{equation}
\label{expansion-i}
\dot\gamma_{i,p}(t)=(d\phi_p)_pv_i+o(t),\quad i=1,2
\end{equation}
so that we obtain
$$
\cos\alpha(t)=\cos\alpha+o(t).
$$
In order to pass from this identity to an estimate on the angle $\alpha(t)$,
we write
$$\alpha(t)=\arccos(\cos\alpha(t)).
$$
The function $\arccos$ is smooth in a neighborhood of $\cos\alpha$ provided
$\sin\alpha\neq 0$, that is, as long as the initial angle is non-degenerate. Assuming $\alpha\in(0,\pi)$, a Taylor expansion of $\arccos$ around
$\cos\alpha$ gives
$$
\arccos(x)=\alpha-\frac{1}{\sin\alpha}\,(x-\cos\alpha)+O\big((x-\cos\alpha)^2\big).
$$
Substituting $x=\cos\alpha(t)$ yields
$$
\alpha(t)=\alpha+o(t),
$$
and hence we correspondingly have
%the corresponding turning the curve of turning angles is 
$$
\theta(t):=\pi-\alpha(t)=\theta+o(t).
$$

The previous equations imply that even if the two curves $\gamma_{i,p}$ fail to be geodesics in $ \R^2$, and the angle function $\alpha(t)$ is not constant, its variation can be read at the second order.
%
%\smallskip
\begin{oss}
    In the case where $S$ is the $2$--dimensional sphere $\mathbb{S}^2$ and $\phi_p$ is the normalized stereographic projection with pole at $p$, geodesics passing through $p$ are mapped to geodesics in $\R^2$. Hence, in this case, we always have $\theta(t)\equiv \theta$.
\end{oss}
\par\noindent
{\em Step 2.}
We now estimate the angles $\beta_i(t)$ that the two curves $\gamma_{i,p}$ make at any time $t\in[0,\delta]$ with the vector $\gamma_{1,p}(t)-\gamma_{2,p}(t)$.
Recalling that $\gamma_{i,p}(0)=\phi_p(p)=(0,0)$, equation \eqref{expansion-i} gives
\begin{equation}
\label{gamma-i-p}
\gamma_{i,p}(t) = (d\phi_p)_p\,v_i\, t + o(t^2),\quad i=1,2.
\end{equation}
In the 
following lines, for the sake of readability, we denote $\gamma_{i,p}$ by $\gamma_i$. 
\noindent
We have
    $$\lVert \dot{\gamma}_i\rVert^2
= \lVert (d\phi_p)_p\,v_i\rVert^2  + o(t)
= 1 + o(t),\quad \lVert \dot{\gamma}_i\rVert^{-1}
=  1  + o(t),\quad
\frac{\dot{\gamma}_i}{\lVert \dot{\gamma}_i\rVert}
=  (d\phi_p)_p\,v_i+ o(t).$$
Moreover,
 $$\gamma_2(t)-\gamma_1(t)
= ((d\phi_p)_p\,v_2-(d\phi_p)_p\,v_1)\, t + o(t^2)$$
\begin{equation}
\lVert \gamma_2(\delta)-\gamma_1(\delta)\rVert=2\sin(\alpha/2)\delta+o(\delta)
\end{equation}
$$
\lVert \gamma_2-\gamma_1\rVert^{-1}
= \frac{1}{\lVert (d\phi_p)_p\,v_2-(d\phi_p)_p\,v_1\rVert t }
\left(
1 + o(t)
\right).$$ Finally,
$$\frac{\gamma_2-\gamma_1}{\lVert \gamma_2-\gamma_1\rVert}
= \frac{(d\phi_p)_p\,v_2-(d\phi_p)_p\,v_1}{\lVert (d\phi_p)_p\,v_2-(d\phi_p)_p\,v_1\rVert }
+ o(t).$$
The angles $\beta_1(t),\beta_2(t)$ are given by
$$\cos\beta_i(t) \coloneqq (-1)^i\,\frac{\gamma_2-\gamma_1}{\lVert \gamma_2-\gamma_1\rVert}\cdot
\frac{\dot{\gamma}_i}{\lVert \dot{\gamma}_i\rVert}
=
\frac{1}{\lVert (d\phi_p)_p\,v_2-(d\phi_p)_p\,v_1\rVert}
\Bigl\{1-\langle v_1,v_2\rangle_{g}+o(t)\Bigr\}\,,$$
for $i=1,2$.
Moreover, we observe that
$$
\arccos\!\left(\frac{1-\langle v_1,v_2\rangle_{g}}{\lVert (d\phi_p)_p\,v_2-(d\phi_p)_p\,v_1\rVert}\right)= \arccos\!\left(\sin\alpha/2\right),
$$
where
$$
\arccos\!\left(\sin\alpha/2\right)=\dfrac{\pi}{2}-\dfrac{\alpha}{2}=\dfrac{\theta}{2}.
$$
Then, we obtain
\begin{equation}\label{beta_i}
    \beta_1(t) = \dfrac{\theta}{2} + o(t),\quad
\beta_2(t) = \dfrac{\theta}{2} + o(t) .
\end{equation}
%
%\smallskip
\par\noindent
{\em Step 3.}
For $i=1,2$, let $\Gamma_{i,p,\delta}$ be the affine line with base point $\gamma_{i,p}(\delta)$ in the direction $\dot\gamma_{i,p}(\delta)$. The angle between $\dot\gamma_{1,p}(\delta)$ and $\dot\gamma_{2,p}(\delta)$ is $\alpha+o(\delta)$. If it is equal to $\pi$, that is, $\det(\dot\gamma_{1,p}(\delta),\dot\gamma_{2,p}(\delta)) = 0$,
the red triangle in Figure \ref{figura} is degenerate.

Otherwise, we let $Q$ be the point of intersection between $\Gamma_{1,p,\delta}$ and $\Gamma_{2,p,\delta}$.
It is given by the unique solution $(s_1,s_2)\in\R^2$ to equation
$$
\gamma_{1,p}(\delta)+s_1 \dot\gamma_{1,p}(\delta)=\gamma_{2,p}(\delta)+s_2\dot\gamma_{2,p}(\delta).
$$
Using that $\det(\dot\gamma_{1,p}(\delta),\dot\gamma_{2,p}(\delta)) \not= 0$, by Cramer's rule we obtain
$$
s_1=\dfrac{\det\left(\gamma_{1,p}(\delta)-\gamma_{2,p}(\delta),\dot \gamma_{2,p}(\delta)\right)}{\det(\dot \gamma_{1,p}(\delta),\dot\gamma_{2,p}(\delta))},\qquad s_2=\dfrac{\det\left(\gamma_{1,p}(\delta)-\gamma_{2,p}(\delta),\dot \gamma_{1,p}(\delta)\right)}{\det(\dot \gamma_{1,p}(\delta),\dot\gamma_{2,p}(\delta))}.
$$
Using the Taylor expansions of the curves $\gamma_{i,p}(t)$ from \eqref{gamma-i-p},
%$$  (d\phi_p)_p\,v_i\, t + o(t^2),\quad i=1,2, $$ 
we have
$$
s_i(t)=\dfrac{t\sin(\alpha)+o(t)}{\sin(\alpha)+o(t)}=t+o(t)\,,\quad i=1,2.
$$
As a consequence, since the length of the edges of the red triangle in Figure \ref{figura} are
$$\bar \ell_i=\lVert Q-\gamma_{i,p}(\delta)\rVert=s_i(\delta)\lVert\dot\gamma_{i,p}(\delta)\rVert,
$$
we obtain the relation
\begin{equation}
 \bar \ell_i=\delta+o(\delta) .
\end{equation}
\subsection{The curve $\gamma_p(P)$}\label{sec:costruzione_curva}
If the angle between $\dot\gamma_{1,p}(\delta)$ and $\dot\gamma_{2,p}(\delta)$ is equal to $\pi$, and hence the red triangle in Figure \ref{figura}
is degenerate, we let $\gamma_p(P)$ be the segment with end points $\gamma_{1,p}(\delta)$ and $\gamma_{2,p}(\delta)$. In this case, we have
$$
\lvert k_{\R^2}(\gamma_p(P))\rvert=0.
$$

Otherwise, with the notation previously introduced, without loss of generality we can suppose that $s_1(\delta)\leq s_2(\delta)$. Setting $t_0=s_2(\delta)-s_1(\delta)$,
% be the difference of the lengths of the two edges.
the curve $\gamma_p(P)$ is parameterized by
$$
\gamma_p(P)(t)=\begin{cases}    \gamma_2(\delta)+t\dot\gamma_2(\delta)& t\in[0,t_0],\\
\Gamma_p(P)(t) &t\in[t_0,t_0+R\theta] ,
      \end{cases}
$$
where $\Gamma_p(P)$ is a piece of a circumference of radius
$$
R=\dfrac{\ell}{\tan\left(\dfrac{\pi-\alpha(\delta)}{2}\right)},
$$
in such a way that $\gamma_p(P)$ is of class $C^1$, with unit tangent vectors
equal to $\dot\gamma_2(\delta)/|\dot\gamma_2(\delta)|$ at the initial point, and to
$\dot\gamma_1(\delta)/|\dot\gamma_1(\delta)|$ at the final point.
Recall that we have:
$$\begin{cases}
    \ell=\tfrac{1}{2}\min\Big\{\bar\ell_1,\bar\ell_2\Big\}=\delta/2 +o(\delta)
\\
(\pi-\alpha(\delta))/2=\theta/2+o(\delta)\,.
\end{cases}
$$
The length of the curve is $$\mathcal{L}(\Gamma_p(P))=R\theta(\delta),\qquad \pi-\alpha(\delta)=\theta(\delta),\qquad \theta(\delta)=\theta+o(\delta).$$
Therefore, we have $$\lvert k_{\R^2}(\gamma_p(P))\rvert^p=\begin{cases}
    0 & t\in[0,t_0],\\
R^{-p} &t\in[t_0,t_0+R\theta]
\end{cases},$$
and the expansion of the radius:
$$R=\dfrac{\delta+o(\delta)}{2\tan(\theta/2)+o(\delta)}.$$
\subsection{The curve $\gamma(P)$}\label{sec:gammaP}
We are in position to define the piecewise smooth curve $\gamma(P)$ associated to an open polygonal $P$, given by the junction of $k$ consecutive and oriented geodesic arcs $\gamma_i$ meeting at the
% with edge length
%$$ 
%{\mathcal L}(\gamma_i)=\ell\quad\forall\,i=1,\ldots,k
%$$
vertices $p_i$, with corresponding turning angles $\theta_i$, for $i=1,\ldots,k-1$.
Without loss of generality, we shall tacitly assume that $0<\theta_i<\pi$ for every $i$.

If $P$ is equilateral, and the edge length of $P$ is sufficiently small, for every $i$ we consider half of the arcs meeting at $p_i$, and construct the curve $\gamma_{p_i}(P)$ through the map $\phi_{p_i}$, so that the estimates obtained in the previous steps hold true.
%We call $\ell=\bar\ell/2$, see \eqref{lato},  around $p_i$. 
If $P$ is not equilateral, but its mesh is sufficiently small, for every $i=0,\ldots,k-1$ we let $\ell_i=\min\{\mathcal L(\gamma_i),\mathcal L(\gamma_{i+1})\}$, and consider the pieces of the two arcs $\gamma_i$ and $\gamma_{i+1}$ with one end point given by $p_i$ and length equal to $\ell_i/2$. We then proceed as above to define the curves $\Gamma_i\coloneqq \phi_{p_i}\circ\gamma_{p_i}(P)$, for $i=1,\ldots,k-1$.
Finally, we define $\gamma(P)$ by connecting the end point of $\Gamma_i$ to the initial point of $\Gamma_{i+1}$ with a piece of the arc $\gamma_{i+1}$, for every $i$.
%The curve $\gamma(P)$ is given by the junction of the curves $\Gamma_i:=\phi_{p_i}%^{-1}\circ\gamma_{p_i}(P)$, for $i=1,\ldots,k-1$.
Notice that, if the polygonal is equilateral, the end points of $\Gamma_i$ are the middle points of the geodesic arcs $\gamma_i$ and $\gamma_{i+1}$. The curve this way obtained is of class $C^1$ and is $C^2$ outside the middle points of the arcs $\gamma_i$. If the polygonal is not equilateral, the curve obtained is still of class $C^1$ and is $C^2$ outside a finite set of points, where the gluing of curves is made.
Therefore, the definition of $p$-rotation in \eqref{p-rotation} makes sense.

By the previous construction, and on account of Theorem \ref{isothermal _coordinates},
Lemma \ref{LemmaComparison} and eq.  \eqref{intcurv}, we 
obtain the following
%More precisely, we recall that the positive radius $\delta_1(\epsilon,S)$ is given by %Theorem \ref{isothermal _coordinates}.
%We have:
%
\begin{prop}\label{PgammaP}
%For every $\epsilon>0$, there exists $\delta_1=\delta_1(\epsilon,S)>0$ such that 
If the mesh of the polygonal $P$ is small,
%$$\ell\leq \delta\leq\delta_1(\epsilon,S) \,, $$
the corresponding curve $\gamma(P)$ satisfies
\begin{equation}
\label{intgammap}
\int\limits_{\gamma(P)}\lvert k_S(\gamma(P))\rvert_g^p\, d{s}=\sum_{i=1}^{k-1}\int\limits_{\gamma_{p_i}(P)}e^{(1-p)\lambda_{p_i}(\gamma_{p_i}(P)(s))}\left\vert k_{\R^2}(\gamma_{p_i}(P)(s)) +\partial_{\hat\gu}\lambda_{p_i}(\gamma_{p_i}(P)(s)) \right\vert^p
\, ds \,,
\end{equation}
where for $i=1,\ldots,k-1$, we have 
$$
\begin{aligned}
&\displaystyle\int_{\gamma_{p_i}(P)}\lvert k_{\R^2}(\gamma_{p_i}(P))\rvert^p\, ds=R_i^{1-p}\theta_i(1+o(\delta)),\\  &
R_i=\dfrac{\delta+o(\delta)}{2\tan(\theta_i/2)+o(\delta)},
\\
 &R_i \theta_i\,(1+o(\delta))\leq \delta +o(\delta),\\ &
\Big|\partial_{\hat\gu}\lambda_{p_i}(\gamma_{p_i}(P)(s))\Big|\leq C\delta.
\end{aligned}
$$
\end{prop}
\section{The $p$-rotation of inscribed polygonals}

In this section, we let \(c:[0,L]\to S\) be a rectifiable open curve
parameterized by arc length.
If \(c\in W^{2,p}([0,L];S)\), then every projected piece \(\Gamma_{\vert_J}\) of \(c\), defined on a subinterval \(J\subset [0,L]\), belongs to \(W^{2,p}(J;\mathbb{R}^2)\). Moreover, by the condition \eqref{stimeIsoterme}, there exists a positive constant \(C=C(c,\epsilon)>0\), such that
\[
\|k_{\mathbb{R}^2}(\Gamma_{\vert_J})\|_{L^p}\leq C.
\]
In particular, for every \(q\in (0,p)\), 
we have
\begin{equation}\label{ACstima}
\int_J |k_{\mathbb{R}^2}(\Gamma_{\vert_J}(t))|^{q}\,dt
\leq \delta^{1-q/p} C^q,
\qquad {{\text{if}}\quad }
|J|\leq \delta.
\end{equation}
Moreover, recall that the positive radius \(\delta_1(\epsilon,S)\) is the one provided by Theorem~\ref{isothermal _coordinates}.
Finally, given an inscribed polygonal curve \(P\ll c\), we shall tacitly use the notation introduced above.
In this section we prove the following
\begin{thm}\label{Teo1}
Let $c\in W^{2,p}([0,L],S)$
be parameterized in arc-length.
For every small $\epsilon>0$, there exist a positive constant $C_\epsilon$ and a small error $g(\epsilon)$, satisfying the bounds
$$
|C_\epsilon-1|\leq\epsilon\,,\quad g(\epsilon)\xrightarrow[]{\epsilon\rightarrow 0}0\,,
$$
such that for every equilateral polygonal $P$ inscribed in $c$, with modulus
$$\mu_c(P)=\delta\leq %\min\{ \delta_{AC}(c,\epsilon),\,
\delta_1(\epsilon,S),%\}\,,
$$
we can estimate
\begin{equation}
\label{stima_energia_poly}
\mathbf{k}_p(P) :=    \int_{\gamma(P)}\vert k_{S}(\gamma(P)(s)) \vert_g^p\, ds\leq C_\epsilon\sum_{i=1}^{k-1} \delta^{-p}\theta_i^{p}(\delta+o(\delta))+g(\epsilon).
    \end{equation}
\end{thm}
The proof is divided in three step. We first obtain a suitable bound on the number of vertices of $P$, and then a bound on the turning angles. Finally, we obtain the estimate in \eqref{stima_energia_poly}.
\begin{oss}
    When \(S\) is the $2$--dimensional sphere \(\mathbb{S}^2\) and \(\phi_p\) is the normalized stereographic projection with pole at \(p\), the construction introduced in \cite{MSS25} differs slightly from the one presented in Section~\ref{sec:gammaP}. Nevertheless, Theorem~\ref{Teo1} provides the same quantitative estimate needed to prove Theorem~\ref{MainThm}. 
\end{oss}
\subsection{A bound on the number of edges}
 The inscribed polygonal being equilateral, the number $k$ of edges of $P$ satisfies
\begin{equation}
\label{kd}
k\delta\leq L\,,
\end{equation}
where $L$ is the length of the curve. If the polygonal is not equilateral, but we have a lower bound on the length of the sides (e.g. they are at least $\frac{\delta}2$), then a similar estimate holds true, replacing $L$ with an absolute constant $C$, only depending on $c$.
\subsection{A bound on the turning angles}
\label{Sec:anglebound}
Let $0=t_0<t_1<\ldots<t_{k-1}<t_k=L$ be such that $p_i=c(t_i)$,
for $i=1,\ldots,k-1$, 
is the $i$-th vertex of the inscribed equilateral polygonal $P$, with turning angle $\theta_i$.
Using arguments from \cite{Bruckstein01112001}, we show that
\begin{equation}
\label{stima_angolo-lato}
\theta_i\leq C\delta^{(p-1)/p}(1+o(\delta))\quad\forall\,i=1,\ldots,k-1\,.
\end{equation}
\par Let $\ell$ be the edge length of $P$. We choose a reparametrization $\psi:[0,{k}\ell]\rightarrow [0,L]$  such that
letting $s_i=i\,\ell$, 
then $\psi(s_i)=t_i$
for $i=0,\ldots,k$, and $\psi$ is affine in each interval $[s_{i-1},s_{i}]$, with
$$\psi'(s)_{\vert(s_{i-1},s_{i})}=\dfrac{t_{i}-t_{i-1}}{s_{i}-s_{i-1}}\quad\forall\, i\in\{1,\ldots,k\} .$$
Following \cite{Bruckstein01112001}, we have
$$
 \tilde C_3\geq  \int_0^L |k_{S}(c(s))|_g^p\, ds\geq\sum_{i=1}^{{k}-1}\dfrac{1}{\ell}\int_0^\ell\left(
\int_{\psi(s_{i-1}+a)}^{\psi(s_{i}+a)}
|k_{S}(c(s))|_g^p\, ds\right) \,da\,.
$$
 Let $\Gamma_i=\varphi_{p_i}(c_{\vert_{[t_{i-1},t_{i+1}]}})$ and $\Gamma_{a,i}=\Gamma_{ i\vert_{ [\psi(s_{i-1}+a),\psi(s_i+a)]}}$, for $i=1,\ldots,k-1$ and $a\in(0,\ell)$. 
{Applying Lemma~\ref{LemmaComparison} 
{and eq. \eqref{intcurv}}
to the local isothermal coordinates in Theorem
 \ref{isothermal _coordinates}, we compare the geodesic curvature of the restricted curves %$$c_i\coloneqq c_{|_{[\psi(s_{i-1}+a),\psi(s_{i}+a)]}}$$
with the curvature in $\R^2$ of their images, namely
$$
\int\limits_{\psi(s_{i-1}+a)}^{\psi(s_{i}+a)}\vert k_{S}(c(s)) \vert_g^p\, ds=\int\limits_0^{\mathcal{L}(a,i)}e^{(1-p)\lambda_{p_i}(\Gamma_{a,i}(s))}\left\vert k_{\R^2}(\Gamma_{a,i}(s)) +\partial_{\hat\gu_i}\lambda_{p_i}(\Gamma_{a,i}(s)) \right\vert^p
\, ds \,,
%+\dfrac{1}{2}e^{\lambda(\Gamma_i(s))}\Gamma_i(s)\cdot \tilde{\gu}_i(s)\right\vert^p
$$
where {  $\mathcal{L}(a,i)\coloneqq\mathcal{L}\left({\Gamma_{a,i}}\right)\leq \tilde C_\epsilon\delta$.}
For every $i=1,\ldots,k-1$ and $a\in(0,\ell)$, thanks to Proposition~\ref{stima_p}, we have
$$\begin{array}{c}
\displaystyle\int_0^{\mathcal{L}(a,i)}e^{(1-p)\lambda_{p_i}(\Gamma_{a,i}(s))}\left\vert k_{\R^2}(\Gamma_{a,i}(s)) +\partial_{\hat\gu_i}\lambda_{p_i}(\Gamma_{a,i}(s)) \right\vert^p
\, ds
%\underset{Prop. \ref{stima_p}}
{\ \geq} \\
\qquad \qquad \displaystyle \tilde{C}_\epsilon\int_0^{\mathcal{L}(a,i)}|k_{\R^2}(\Gamma_{a,i})|^p\,ds  -C_0\delta\int_0^{\mathcal{L}(a,i)}|k_{\R^2}(\Gamma_{a,i})|^{p-1}\,ds\,,
\end{array}
$$
where $C\coloneqq C(\epsilon,p)$ is a positive constant that is bounded as $\epsilon\rightarrow 0$, that can change line by line, and $\tilde{C}_\epsilon\to 1$ as $\epsilon\to 0$.   We denote by $I_1,I_2$ the two terms
$$
\begin{array}{c}
\displaystyle
I_1\coloneqq\sum_{i=1}^{k-1}\tilde{C}_\epsilon\dashint_0^\ell\left(\int_0^{\mathcal{L}(a,i)}|k_{\R^2}(\Gamma_{a,i}(s))|^p\,ds\right)\,da, \\
\displaystyle
I_2\coloneqq\sum_{i=1}^{k-1}C\delta\dashint_0^\ell \left(\int_0^{\mathcal{L}(a,i)}|k_{\R^2}(\Gamma_{a,i}(s))|^{p-1}\,ds\right)\,da \,.
\end{array}
$$
In the Euclidean setting, since the curve {$\Gamma_{a,i}$} is parameterized by arc length, then its scalar curvature is given by the norm of its second derivative 
{$\ddot\Gamma_{a,i}(s)$}, i.e.,
 $$
\int_0^{\mathcal{L}(a,i)}|k_{\R^2}(\Gamma_{a,i}(s))|^p\, ds=\int_0^{\mathcal{L}(a,i)}\norm{\ddot\Gamma_{a,i}(s)}^p\, ds\,.
$$

Applying twice the Jensen inequality and supposing $t_{i+1}-t_i\geq t_i-t_{i-1}$, we have
$$
\begin{aligned}     \tilde{C}_\epsilon\dashint_0^\ell   \Biggl(\int_0^{\mathcal{L}(a,i)}&\norm{\ddot\Gamma_{a,i}(s)}^p\, ds\Biggr)\,da\geq\tilde{C}_\epsilon\dashint_0^\ell \| \dot\Gamma_{a,i}(\psi(s_{i}+a))-\dot\Gamma_{a,i}(\psi(s_{i-1}+a)) \|^p\, da \\
      & \geq  \tilde{C}_\epsilon\ell^{1-p}\left\|\int_0^\ell\left(\dot\Gamma_{a,i}(\psi(s_{i}+a))-\dot\Gamma_{a,i}(\psi(s_{i-1}+a))\right)\, da\right\|^p \\   
      & \geq\tilde{C}_\epsilon\ell^{1-p}\left\|  \dfrac{\Gamma_{a,i}(t_{i+1})-\Gamma_{a,i}(t_i)}{t_{i+1}-t_i}- \dfrac{\Gamma_{a,i}(t_{i})-\Gamma_{a,i}(t_{i-1})}{t_{i}-t_{i-1}}\right\|^p.
   \end{aligned}$$}
   Finally, using \cite[Theorem 5.3]{MS1}, we have
$$
\sum_{i=1}^{k-1}\tilde{C}_\epsilon\dashint_0^\ell \left(
\int_0^{\mathcal{L}(a,i)}\norm{\ddot\Gamma_{a,i}(s)}^p\, ds \right)\,da \geq \sum_{i=1}^{k-1} \tilde{C}_\epsilon\,\tilde\ell_i^{1-p}\,\tilde\theta_i^p,
$$  where $\tilde\ell_i$ and $\tilde\theta_i$ are the length and the turning angle of the Euclidean polygonal with two edges inscribed in 
$[\Gamma_i(t_{i-1}),\Gamma_i(t_{i}),\Gamma_i(t_{i+1})].$

Applying the inequality in \eqref{ACstima}, we have
$$\int_0^{\mathcal{L}(a,i)}\left\vert k_{\R^2}(\Gamma_{a,i}(s))\right\vert^{q}\,ds \leq C\delta^{(p-q)/p}.$$
Then, using $q=p-1$ we obtain
$$%\begin{aligned}
    I_2=\sum_{i=1}^{k-1}C\delta \dashint_0^\ell \left( \int_0^{\mathcal{L}(a,i)}|k_{\R^2}(\Gamma_{a,i}(s))|^{p-1}\,ds\right)\,da
\leq
\sum_{i=1}^{k-1}C\delta^{1+1/p} \,.
%\end{aligned} 
$$
Summing over the index $i$ and using the estimate in \eqref{kd} we get
\begin{equation}\label{secondadis}
    \begin{aligned}
  I_2\leq C\delta^{\tfrac{1}{p}} =: g(\epsilon)\,.
\end{aligned}
\end{equation}
By construction, we have
$$
\tilde C_3 \geq I_1-I_2
$$
Then, there exists a constant $C_3$ depending only on the curve $c$ and on $\epsilon$ such that
$$
\sum_{i=1}^{k-1} \tilde{C}_\epsilon\,\tilde\ell_i^{1-p}\,\tilde \theta_i^p\leq C_3 .
$$
%\begin{prop}\label{euclll}
Putting all together, we obtain \begin{equation}\label{stima-basso}
    \int_0^L |k_{S}(c(s))|_g^p\, ds\geq \sum_{i=1}^{k-1} \tilde{C}_\epsilon\,\tilde\ell^{1-p}\,\tilde \theta_i^p-g_2(\epsilon).
\end{equation}
\par Now, by \eqref{expansion} we infer that \begin{equation}\label{stima-lati-angolis}
    \tilde\ell_i=\delta+o(\delta),\qquad \tilde\theta_i=\theta_i+o(\delta).
\end{equation}
$$$$
Therefore, for every $i$ we obtain
$$
\delta^{1-p}\theta_i^p(1+o(\delta))\leq C_3,
$$
which implies the validity of the estimate in \eqref{stima_angolo-lato}.
%\begin{equation}\label{stima_angolo-lato}
%    \theta_i\leq C_3 \delta^{1/\bar p}(1+o(\delta)),\qquad \bar p=\dfrac{p}{p-1}.
%\end{equation}
\par Finally, for future use we notice that by property \eqref{stima_angolo-lato} we can assume
$$\tan(\theta_i/2)\leq \frac{\theta_i}{2}(1+\epsilon)\quad\forall\,i=1,\ldots,k-1,
$$
and we recall that by the construction
\begin{equation}
\label{stima_Raggio}
    \mathcal{L}(\gamma_{p_i}(P))=R_i(\theta_i+o(\delta))\leq \delta+o(\delta).
\end{equation}
\subsection{Estimate on the $p$-rotation}
We are now in position to prove the inequality \eqref{stima_energia_poly}.

By Proposition \ref{stima_p} and by the estimate \eqref{stimeIsoterme}, we use
$$
a=\vert k_{\R^2}(\gamma_{p_i}(P)(s))\vert=\dfrac{1}{R_i},\qquad b=\lvert\partial_{\hat\gu}\lambda(\gamma_{p_i}(P)(s))\rvert\underset{\eqref{stimeIsoterme}}{\leq}C_0\delta ,
$$
in the formula \eqref{intgammap}. Using estimate \eqref{dis-alto}, we obtain the upper bound
$$
%\hspace{-0.7cm} 
\mathbf{k}_p(P)\leq  C_\epsilon\sum_{i=1}^{k-1}
\int_{\gamma_{p_i}(P)}
\left(\dfrac{1}{R_i^p} +
C
%\sum_{i=1}^{k-1}\int\limits_{\gamma_{p_i}(P)}
\,\dfrac{C_0\delta}{R_i^{p-1}} %\,ds
%+\tilde{C}
%\sum_{i=1}^{k-1}\int\limits_{\gamma_{p_i}(P)}\,\dfrac{C_0^{p-1}\delta^{p-1}}{R_i} %\,ds
+
%\sum_{i=1}^{k-1}\int\limits_{\gamma_{p_i}(P)}
CC_0^p\delta^{p}\right)\,ds \,.
$$
We integrate the quantities on $\gamma_{p_i}(P)$. Using the length estimate in \eqref{stima_Raggio}, we have
$$
%\hspace{-0.7cm}
\mathbf{k}_p(P) \underset{\eqref{stima_Raggio}}{\leq}
C_\epsilon\sum\limits_{i=1}^{k-1}
\Biggl( \dfrac{{\theta_i}(1+o(\delta))}{R_i^{p-1}}
+C%\sum_{i=1}^{k-1}
\,\dfrac{C_0\delta^2(1+o(1))  }{R_i^{p-1}} +%\sum_{i=1}^{k-1}
CC_0^p\delta^{p+1}(1+o(1)) \Biggr)\,,
$$
where $o(1)$ denotes a quantity depending on $\delta$ that goes to zero as $\delta\to 0$.

Now, we deal with each term in the previous line. The constant $ C$ can change line by line.
%\begin{itemize}    \item 
The first term is
$$ %\sum_{i=1}^{k-1}\int_{\gamma_{p_i}(P)}\dfrac{1}{R_i^p}\,ds= 
\sum_{i=1}^{k-1}\dfrac{{\theta_i}(1+o(\delta))}{R_i^{p-1}}= \sum_{i=1}^{k-1}\dfrac{(2\tan({\theta_i}/2)+o(\delta))^{p-1}\theta_i(1+o(1))}{\delta^{p-1}(1+o(1))}\leq \sum_{i=1}^{k-1} \delta^{1-p}\theta_i^{p}(1+o(1)).
$$
%\item 
The second one is
$$
\begin{aligned}
%    \tilde{C}\sum_{i=1}^{k-1}\int_{\gamma_{p_i}(P)}\dfrac{C_0\delta}{R_i^{p-1}}\,ds&\underset{\eqref{stima_Raggio}}{\leq}
C\sum_{i=1}^{k-1}\dfrac{C_0\delta^2(1+o(1))}{R_i^{p-1}}=C\sum_{i=1}^{k-1}\dfrac{C_0\delta^2(2\tan({\theta_i}/2)+o(\delta))^{p-1}(1+o(1))}{\delta^{p-1}(1+o(1))}
\\ %& 
\qquad\qquad
\underset{\eqref{stima_angolo-lato}}{\leq} C\sum_{i=1}^{k-1}\delta^2\delta^{\frac{ {(p-1)^2} }{p}}\delta^{1-p}(1+o(1))= C\sum_{i=1}^{k-1}\delta^{1+\frac{1}{p}}(1+o(1))
\underset{\eqref{kd}}{\leq} C\delta^{\frac{1}{p}}(1+o(1)).
\end{aligned}
$$
%\item  
The last term is $$CC_0^p\sum_{i=1}^{k-1}\int_{\gamma_{p_i}(P)}\delta^{p}\,ds\underset{\eqref{stima_Raggio}}{\leq}C\sum_{i=1}^{k-1}\delta^{p+1}(1+o(1))
\underset{\eqref{kd}}{\leq}  C \delta^{p}(1+o(1)).$$
%\end{itemize}
In conclusion, the inequality \eqref{stima_energia_poly} holds true, where
$$
g(\epsilon)=C\left(\delta^{\frac{1}{p}}(1+o(1))+ \delta^{p}(1+o(1))\right)\leq C\delta^{\frac{1}{p}}(1+o(1))
$$
and Theorem \ref{Teo1} is proved.
%
%%%%%%%%%%%%
%
\section{Proof of the Main Result}
Similarly to the case of spherical curves analyzed in \cite{MSS25}, the proof of Theorem \ref{MainThm} is divided in two parts.
We first show, Theorem \ref{thm1}, that if a curve $c$ has finite $p$-energy, then the expected second order Sobolev regularity holds and we obtain the upper bound estimate of the integral of the $p$-power of the norm of the geodesic curvature.
Then, Theorem \ref{lowerbound}, using the previous results we prove the lower bound estimate. Theorem  \ref{MainThm} readily follows.
\subsection{Upper bound estimate}
Arguing essentially as in the proof of Theorem 5.3 from \cite{MSS25}, we obtain the following
\begin{thm}\label{thm1}
 Let $c$ be a rectifiable and open curve in $S$ parameterized by arc length such that $\mathcal{F}_p(c)<\infty$ for some $p>1$. Then $c \in W^{2,p}(I_L,S)$, where $I_L=[0,L]$ and $L={\mathcal L}(c)$, and
$$
\int_0^L|k_S(c)|_g^p\, ds\leq \mathcal{F}_p(c)<\infty.
$$
\end{thm}
\begin{proof}
%It is a straightforward readaptation of the one of Theorem 5.3 from \cite{MSS25}, where this time the curve $\gamma(P_h)$ is associated to the inscribed polygonal $P_h$ as in the previous construction, see Remark \ref{R:pol}, and the term ...
Let $\{P_h\}$ be a sequence of polygonal curves inscribed in $c$ satisfying $\mu_c(P_h) \rightarrow 0$ and $\mathbf{k}_p(P_h) \rightarrow \mathcal{F}_p(c)$. For each $h$, let $c_{P_h} : [0, L_h] \to S$ be the arc-length parametrization of the piecewise smooth and $\mathcal{C}^1$ curve $\gamma(P_h)$ obtained in  Proposition \ref{PgammaP}, where $L_h \coloneqq\nobreak \mathcal{L}(\gamma(P_h))$.
We first notice that
\begin{equation}
\label{LhL}
\lim_{h\to \infty} L_h=L\,.
\end{equation}

In fact, we have $d_{F}(\gamma(P_h),P_h)\leq\mu_c(P_h)$ for every $h$, whereas $d_{F}(P_h,c)\rightarrow 0$. Since
$\mu_c(P_h) \rightarrow 0$, we obtain $d_{F}(\gamma (P_h), c) \rightarrow 0$, and hence by lower semicontinuity we infer that
$$
\mathcal{L}(c)\leq\liminf_{h\rightarrow +\infty}{\mathcal{L}(\gamma(P_h))} \,.
$$
Using the fact that $\mathcal{L}(\gamma (P_h )) \leq \mathcal{L}(P_h ) \leq \mathcal{L}(c)$ for every $h$, we deduce the limit in \eqref{LhL}.

Since it is easier to make computations on extrinsic objects, we let
$$ \gc_{P_h}:={\bf j}\circ c_{P_h} : [0, L_h] \to {\bf j}(S)\subset \R^m , $$
and let $ \gamma_h : [0, L]\rightarrow {\bf j}(S)$ be given by
$$
\gamma_h(s)\coloneqq \gc_{P_h}(sL_h/L)\,.
$$
Notice that by the construction it turns out that each $\dot \gamma_h$ is in the Sobolev space $W^{1,1}(I_L, \R^m)$.

By piecewise smoothness, apart from a finite set of points,
one has
%$k_{\mathbb{S}^2}(c_{P_h}(\lambda))$
$$\norm{ k_{{\bf j}(S)}(c_{P_h}(\lambda)) } = \norm{\Ddot{\gc_{P_h}}(\lambda)^\top}$$
for $\lambda \in [0, L_h]$ and
 $$\Ddot{\gamma}_h(s)^\top =
    (L_h/L)^2\Ddot{\gc_{P_h}}(\lambda)^\top\,,\quad
\Ddot{\gamma}_h(s)^\perp =
    (L_h/L)^2\Ddot{\gc_{P_h}}(\lambda)^\perp
$$
for $s \in I_L$, with $\lambda = s(L_h/L)$.
Therefore, we can write
\begin{equation}\label{Kp}
        \mathbf{k}_p(P_h)\coloneqq \int_{\gamma(P_h)} |k_{S}(\gamma(P_h))|_g^p\,ds=
\int_0^{L_h}\norm{  \Ddot{\gc_{P_h}}(\lambda)^\top}^p\,d\lambda =
\left(\frac L{L_h}\right)^{2p-1} \!\! \int_0^{L}\norm{  \Ddot{\gamma}_h(s)^\top}^p\,d s\,.
    \end{equation}
As a consequence, recalling that $\mathbf{k}_p(P_h) \rightarrow \mathcal{F}_p(c)$, by \eqref{LhL} and \eqref{Kp} we obtain
\begin{equation}
\label{limit}
       \lim_{h\rightarrow \infty } \int_0^{L}\norm{  \Ddot{\gamma}_h(s)^\top}^p\,d s=\mathcal{F}_p(c) \,.
\end{equation}

Now, since $p > 1$, the sequence $\{\Dot{\gamma}_h\}$ converges strongly in $W^{1,1}$ to some function $v \in W^{1,1}(I_L, \R^m)$. By using that $\{ \gamma_h\} $ converges to the Lipschitz function $\gc={\bf j}\circ c$ strongly in $L^1(I_L, \R^m)$, we obtain $v = \Dot{\gc}$ a.e.; hence, possibly passing to a (not relabeled) subsequence, $\{\Dot{\gamma}_h\}$ converges to $\Dot{\gc}$ weakly in $W^{1,p}(I_L, \mathbb{R}^m)$. In particular, $\Dot{\gc} \in W^{1,p}(I_L,\R^m) $ and
$\{ \Ddot{\gamma}_h\}$ converges to $\Ddot{\gc}$ weakly in $L^p$.
Therefore, the curve $\gc$ is in $W^{2,p}(I_L,{\bf j}(S))$ and hence  $c\in W^{2,p}(I_L,S)$.

Furthermore, using that both the convergence of $\gamma_h$ to $\gc$ and of $\dot\gamma_h$ to $\dot\gc$ are uniform, we infer that for every $s$ the tangent planes at $\gamma_h$ converge to the tangent plane at $c(s)$ and the same for the conormal. We thus deduce that 
$\{ \Ddot{\gamma}^\top_h\}$ converges to $\Ddot{\gc}^\top$ almost everywhere and hence weakly in $L^p$, as $h\rightarrow \infty$.

%Since the normal component $\Ddot{\gc}^\perp$ is uniformly bounded by a constant only %depending on $S$, we infer that the curve $\gc$ is in $W^{2,p}(I_L,{\bf j}(S))$ and hence %that $c\in W^{2,p}(I_L,S)$. 
Finally, by lower semicontinuity, using \eqref{limit} we obtain
$$
\int_0^L|k_S(c)|_g^p\, ds = \int_0^L\norm{\Ddot{\gc}^\top(s)}^p\, ds\leq
\liminf_{h\rightarrow \infty }\int_0^{L}\norm{  \Ddot{\gamma}_h(s)^\top}^p\,d s
=\mathcal{F}_p(c)
$$
and the proof is complete.
\end{proof}
\subsection{Lower bound estimate}
Using arguments from the previous sections, we readily obtain the following
\begin{thm}\label{lowerbound}
Let $c$ be a rectifiable and open curve in $S$ parameterized by arc length and of class  $W^{2,p}([0,L],S)$ for some $p>1$.
Then,  for every $\epsilon>0$ small, there exists a polygonal $P_\epsilon$ inscribed in the curve $c$ and a
positive constant $C_\epsilon$ such that $\mu_c(P_\epsilon) \rightarrow 0 $, $C_\epsilon\rightarrow 1$
as $\epsilon\rightarrow 0$ and
\begin{equation}
\label{lbdP}
 \tilde{C}_\epsilon(1+g_0(\epsilon))\mathbf{k}_p(P_{\epsilon})+g_2(\epsilon)-g_1(\epsilon)  \leq \int_0^L\vert k_{S}(c(s)) \vert_g^p\,ds<\infty,
\end{equation}
where $g_i(\epsilon)\rightarrow 0$ as $\epsilon\rightarrow 0$, for $i=0,1,2$.
Therefore, we have
\begin{equation}
\label{lowerboundd}
\mathcal{F}_p(c)\leq\int_0^L\vert k_{S}(c(s))\vert_g^p\, ds.
\end{equation}
\end{thm}
% £
\begin{proof}
Fix $\epsilon>0$ small.
In Theorem \ref{Teo1} we have seen that
every equilateral polygonal $P$ inscribed in $c$, with small modulus, satisfies the $p$-rotation upper bound in equation \eqref{stima_energia_poly}, that is
$$
\mathbf{k}_p(P) \leq C_\epsilon\sum_{i=1}^{k-1} \delta^{-p}\theta_i^{p}(\delta+o(\delta))+g(\epsilon),
$$
where $g(\epsilon)\to 0$ as $\epsilon\to 0$.
Since moreover $\delta\to 0$ as $\epsilon\to 0$, we also have $o(\delta)=g_0(\epsilon)$.

Furthermore, using the computation with a reparametrization $\psi$ made in Section \ref{Sec:anglebound}, we have obtained  the inequality
$$  \int_0^L |k_{S}(c(s))|_g^p\, ds\underset{\eqref{stima-basso}}{\geq} \sum_{i=1}^{k-1} \tilde{C}_\epsilon\,\tilde\ell^{1-p}\,\tilde \theta_i^p-g_2(\epsilon)\underset{\eqref{stima-lati-angolis}}{\geq} \tilde{C}_\epsilon\sum_{i=1}^{k-1}\delta^{1-p}\theta_i^p(1+o(\delta))-g_2(\epsilon)\,.
$$

Putting the terms together, we find $P_\epsilon\ll c$ for which the inequality \eqref{lbdP} holds.
Finally, the inequality \eqref{lowerboundd} follows taking the sequence $\{P_{\epsilon_h}\}$, where $\epsilon_h\rightarrow 0^+$ and for each $\epsilon_h$ the polygonal $P_{\epsilon_h}$ is the one constructed before.
Further details are omitted.
\end{proof}

%\bibitem[R05]{Re}
%Yu.~Reshetnyak.
%\newblock The theory of curves in differential geometry from the viewpoint of
%  the theory of functions of a real variable.
%\newblock {\em Russian Math. Surveys}, 60:1165--1181, 2005.

%\bibitem[S06]{Su}
%John Sullivan.
%\newblock Curves of finite total curvature.
%\newblock {\em Discrete Differential Geometry}, 38:137--161, 2006.}


\begin{thebibliography}{ZZZ}
\normalsize{
\bibitem[AR12]{alexandrov2012general}V.~V. Alexandrov and Y.~G. Reshetnyak.
\newblock {\em General Theory of Irregular Curves}.
\newblock Mathematics and its Applications. Springer Netherlands, 2012.

\bibitem[BNR01]{Bruckstein01112001}
A.~M. Bruckstein, A.~N. Netravali, and T.~J. Richardson.
\newblock Epi-convergence of discrete elastica.
\newblock {\em Applicable Analysis}, 79(1-2):137--171, 2001.

%\bibitem[dC76]{doCarmo1976}Manfredo~P. do~Carmo.
%\newblock {\em Differential Geometry of Curves and Surfaces}.
%\newblock 1976.

%\bibitem[CFM10]{C}
%M.~Castrill{\'o}n L{\'o}pez, V.~Fern{\'a}ndez Mateos, and J.~Mu{\~n}oz Masqu{\'e}.
%\newblock Total curvature of curves in {R}iemannian manifolds.
%\newblock {\em Differential Geom. Appl.}, 28(2):140--147, 2010.

\bibitem[C55]{Ch} S.~S. Chern. 
\newblock An elementary proof of the existence of isothermal parameters on a surface. \newblock {\em Proc. Amer. Math. Soc.} {\bf 6}:771--782, 1955.

%\bibitem[D79]{D}B.~V. Dekster.
%\newblock Upper estimates of the length of a curve in a {R}iemannian manifold  with boundary.\newblock {\em J. Differential Geom.}, 14(2):149--166, 1979.

\bibitem[DK81]{DK} D.~M. DeTurck and J.~L. Kazdan.
\newblock Some regularity theorems in Riemannian geometry.
\newblock {\em Ann. Sci. \'Ecole Norm. Sup.} (4) {\bf 14} no.~3:249--260, 1981.

%\bibitem[M50]{Mil}J.~W. Milnor.
%\newblock On the total curvature of knots.
%\newblock {\em Ann. of Math.}, 52(2):248--257, 1950.

%\bibitem[M53]{Mi53}J.~W. Milnor.
%\newblock On total curvatures of closed space curves.
%\newblock {\em Mathematica Scandinavica}, 1:289--296, 1953.

\bibitem[MS21a]{MS2}
D.~Mucci and A.~Saracco.
\newblock The total intrinsic curvature of curves in {R}iemannian surfaces.
\newblock {\em Rend. Circ. Mat. Palermo (2)}, 70:521--557, 2021.

\bibitem[MS21b]{MS2b}
D.~Mucci and A.~Saracco.
\newblock Correction to: The total intrinsic curvature of curves in {R}iemannian
  surfaces.
\newblock {\em Rend. Circ. Mat. Palermo (2)}, 70:1137--1138, 2021.

\bibitem[MS23]{MS1}
D.~Mucci and A.~Saracco.
\newblock Weak elastic energy of irregular curves.
\newblock {\em Philos. Trans. Roy. Soc. A}, 381, 2023.

\bibitem[MSS26]{MSS25}
D.~Mucci, A.~Saracco and C.~Sopio.
\newblock Weak elastic energy of rectifiable curves in the sphere.
\newblock {\em Proc. R. Soc. A}, 482, 2026.}
\end{thebibliography}
\end{document}